\documentclass[11point]{amsart}
\usepackage{amsmath}
\usepackage{amscd}

\usepackage{amssymb,amsmath,amscd,epsfig,amsthm,enumerate,import, comment,xcolor}

\usepackage{cancel}

\usepackage{geometry}                
\newtheorem{theorem}{Theorem}[section]
\newtheorem{proposition}[theorem]{Proposition}
\newtheorem{lemma}[theorem]{Lemma}
\newtheorem{corollary}[theorem]{Corollary}
\theoremstyle{definition}
\newtheorem{definition}[theorem]{Definition}

\theoremstyle{remark}
\newtheorem{remark}[theorem]{Remark}

\numberwithin{equation}{section}
\newcommand{\be}{\begin{equation}}
\newcommand{\ee}{\end{equation}}

\newcommand{\bbC}{{\mathbb C}}
\newcommand{\bbZ}{{\mathbb Z}}
\newcommand{\bbR}{{\mathbb R}}
\newcommand{\bbN}{{\mathbb N}}

\newcommand{\calT}{{\mathcal T}}

\newcommand{\calV}{{\mathcal V}}

\newcommand{\calU}{{\mathcal U}}

\newcommand{\calK}{{\mathcal K}}

\newcommand{\calC}{{\mathcal C}}

\newcommand{\calS}{{\mathcal S}}

\newcommand{\x}{\tilde{x}}

\newcommand{\frakp}{{\mathfrak p}}

\newcommand{\norm}[1]{\lVert#1\rVert}

\renewcommand{\Box}{\square}

\newcommand{\wt}{\widetilde}

\newcommand{\h}{\hbar}

\newcommand{\hinv}{\hbar^{-1}}

\begin{document}
	
	\title{Integral representations of isotropic Semi-Classical Functions and applications}
	\author{V. Guillemin}
	\address{Department of Mathematics \\
		Massachusetts Institute of Technology\\Cambridge, MA 02139}
	\email{vwg@math.mit.edu}
	\author{A. Uribe}
	\address{Department of Mathematics \\
		University of Michigan\\Ann Arbor, Michigan 48109}
	\email{uribe@umich.edu}
	\author{Z. Wang}
	\address{School of Mathematical Sciences\\
		University of Science and Technology of China \\ Hefei, Anhui 230026, P.R.China}
	\email{wangzuoq@ustc.edu.cn}
    \thanks{Z.W. is supported in part by National Key Research and Development Project SQ2020YFA070080 and NSFC 12026409, NSFC 11721101.}

	\date{\today}
	\begin{abstract}
In \cite{GUW} we introduced a class of ``semi-classical
functions of isotropic type", starting with a model case and applying Fourier integral
operators associated with canonical transformations.
These functions are a substantial generalization	of the ``oscillatory functions of
Lagrangian type" that have played major role in semi-classical and
micro-local analysis.
In this paper we exhibit more clearly the nature of these isotropic functions by obtaining oscillatory integral
expressions for them.  Then we use these to prove that the classes of isotropic functions are equivariant with respect
to the action of general FIOs (under the usual clean-intersection hypothesis).
The simplest examples of isotropic states are the ``coherent states", a class of oscillatory functions
that has played a pivotal role in mathematics and theoretical physics
beginning with their introduction by of Schr\"odinger in the 1920's.  We prove that  every
oscillatory function of isotropic type can be expressed as a
superposition of coherent states, and examine some implications of that fact.
We also show that certain functions of elliptic operators have isotropic functions for
Schwartz kernels.  This lead us to a result on an eigenvalue counting function that
appears to be new (Corollary \ref{cor:altWeyl}).
	\end{abstract}
	
	\maketitle

\centerline{\em In memory of Mikhail Shubin.}
	
	\tableofcontents

\section{Introduction}

The world of microlocal analysis is populated by objects of ``Lagrangian type". For instance, pseudodifferential operators on a manifold $X$ live microlocally on the diagonal in $T^*X \times (T^*X)^-$, and Fourier integral operators on  a canonical  relation
\[
\Gamma: T^*X \Rightarrow T^*Y 
\]
which is a Lagrangian submanifold of $T^*X \times (T^*Y)^{-}$.
The topic of this paper however will be objects in semi-classical analysis that live instead on \emph{isotropic}
submanifolds, which can be defined as any submanifold of a Lagrangian submanifold.
Microlocal objects of this type are not entirely unfamiliar, and in complex analysis have been around since the 1970's.
A prototypical example is the Szeg\"o projector,
\[
\Pi: L^2(X) \to H(X),
\]
where $X$ is the boundary of a smooth compact pseudoconvex domain, $\Omega$, in $\mathbb C^n$, and $H(X)$ is the $L^2$-closure of $\mathcal O(\Omega)|_X$.
More generally, from $\Pi$ one gets a large class of operators,
\begin{equation}\label{Toeplitz}
T = \Pi P \Pi: H(X) \to H(X),
\end{equation}
where $P$ is a pseudodifferential operator on $X$.
Operators of the form \eqref{Toeplitz}, known as the (generalized) Toeplitz operators, were first defined and studied by Boutet de Monvel and the first named author of this paper in \cite{BG},
and have played an important role in the theory of geometric quantization as well as in other areas.   Their Schwartz kernels are isotropic distributions, associated to
a particular submanifold of the diagonal (which we won't describe here).

In this paper we will consider objects of this type in the real $C^\infty$ world rather than in the complex world,  and within the framework of semi-classical rather than microlocal analysis.  Furthermore, since we are working in the semi-classical world, we no longer need to assume that these isotropic objects be conic, as in the microlocal setting.

The prototypical example of such an isotropic function is
\begin{equation}\label{Rnprototypical}
\Upsilon(x, \hbar) := \hbar^r e^{if(x)/\hbar} \varphi(x', \hbar^{-1/2}x'', \hbar),
\end{equation}
where $\hbar$ is Planck's constant (or the semi-classical parameter) and
\begin{itemize}
\item $x=(x', x'')$ are coordinates on $\mathbb R^n=\mathbb R^k \times \mathbb R^l$,
\item $f \in C^\infty(\mathbb R^n)$ is a  real-valued smooth function,
\item $\varphi \in C^\infty(\mathbb R^n \times \mathbb  (0,\h_0), \bbC)$ is a smooth function which
is Schwartz in the variable $x''$ with estimates locally uniform in $x'$ and
admits an asymptotic expansion
\begin{equation}\label{2exp}
	\varphi(x',x'',\h) \sim \h^{r}\sum_{j=0}^\infty \varphi_j(x',x'')\, \h^{j/2},
\end{equation}
and where, for all $j$, $\varphi_j(x',x'')$ is
a Schwartz function in the $x''$ variable satisfying Schwartz estimates that are locally uniform in $x'$.
\end{itemize}

As $\hbar$ tends to zero, the function $\Upsilon$ becomes more and more concentrated on the manifold
\[
Z=\mathbb R^k \times \{0\}
\]
in configuration space and more and more concentrated on the manifold
\[
\Sigma = \{(x, df_x)\ |\ x \in Z\}
\]
in phase space. Note that $\Sigma$ is an isotropic submanifold inside the Lagrangian submanifold
\begin{equation}\label{lambdaf}
	\Lambda_f=\{(x, df_x)\ |\ x \in \mathbb R^n\},
\end{equation}
which has $f$ as its generating function.

More generally, suppose that $X$ is an $n$-dimensional manifold
and $\Sigma\subset T^*X$ an isotropic submanifold of its
cotangent space.  In \cite{GUW} we defined classes of $\h$-dependent smooth functions on $X$,
$I^r(X,\Sigma)$ ($r$ being a half integer), whose semi-classical wave front set is contained in $\Sigma$.
%
%
Our general definition of $I^r(X,\Sigma)$ is as follows.
We take, as a model case,
the prototypical example above where $f= 0$.  The resulting functions have wave-front set in
the canonical isotropic $\Sigma_0$ given by
\begin{equation}\label{Sigma0}
\Sigma_0 = \{(x, \xi): x''=0, \xi'=0, \xi''=0\} \subset T^*\mathbb R^n,
\end{equation}
and we define  a model class $I^r(\bbR^{n},\Sigma_0)$ whose elements $\Upsilon\in I^r(\bbR^{n},\Sigma_0)$ are of the form
\begin{equation}\label{basicdef}
\Upsilon(x,\h) = \varphi(x', \hbar^{-1/2}x'', \hbar)
\end{equation}
where $\varphi(x',x'',\h)$, as $\h\to 0$, admits an asymptotic expansion
\begin{equation}\label{2exp}
\varphi(x',x'',\h) \sim \h^{r}\sum_{j=0}^\infty \varphi_j(x',x'')\, \h^{j/2},
\end{equation}
and where, for all $j$, $\varphi_j(x',x'')$ is
a Schwartz function in the $x''$ variable satisfying Schwartz estimates that are locally uniform in $x'$.
After proving that the space $I^r(\bbR^{n},\Sigma_0)$ is invariant under FIOs associated with canonical transformations
preserving $\Sigma_0$ (as a set), we were able to define the general classes $I^r(X,\Sigma)$ microlocally as the
images of these functions under FIOs associated to local canonical transformations mapping
$\Sigma_0$ to $\Sigma$.

One of the  main results of the present paper is to prove that these classes are invariant under the action of arbitrary Fourier integral operators, provided the usual clean intersection condition is satisfied:
\begin{theorem}
\label{FIOinvThm}
Let $\Sigma\subset T^*Y$ be an isotropic, and $\Upsilon \in I^r(Y, \Sigma)$ an associated isotropic
function.  Let $F: C_0^\infty(Y)\to C^\infty (X)$ be a semiclassical FIO of order zero
whose canonical relation, $\Gamma \subset T^*X \times (T^*Y)^-$,
intersects $\Sigma$ cleanly.  Then
\[
F(\Upsilon)\in I^{r+(\dim Y -\dim X-e)/2} (X, \Gamma(\Sigma)),
\]
 i.e. $F(\Upsilon)$  is an isotropic function associated with the image of $\Sigma$ under $\Gamma$. Here $e$ is the excess of the clean composition (see \S 3).
\end{theorem}

It turns out that in proving this FIO invariance result, the definition of $I^r(X, \Sigma)$ we alluded to above is not convenient to use.
For this reason, and because of its general interest to this theory, we will
 give another equivalent characterization for  elements in $I^r(X, \Sigma)$ in terms of local integral representations,
 generalizing the Lagrangian case as developed by H\"ormander.

Before describing this H\"ormander approach we will first of all assume that $\Sigma$ is a \emph{horizontal} isotropic submanifold of $T^*X$, i.e. that there exists a submanifold $X_\Sigma$ of $X$ and a function $f \in C^\infty(X)$ such that
\[
(x, \xi) \in \Sigma \Longleftrightarrow x \in X_\Sigma \ \mbox{and}\  \xi=df_x.
\]
Then in this case we will define $I^r(X, \Sigma)$ to be the set of functions $\Upsilon(x, \hbar) \in C^\infty(X \times \mathbb R)$
which, at points $x \in X\setminus X_\Sigma$, vanish to order $O(\hbar^\infty)$ on a neighborhood of $x$ and for $x \in X_\Sigma$ are of the form \eqref{Rnprototypical} on a coordinate patch $\mathbb R^n=\mathbb R^k \times \mathbb R^l$ with the properties above.

Coming back to the H\"ormander's approach let $\Lambda \subset T^*X$ be an arbitrary Lagrangian submanifold. Then, at
least locally, one can define $\Lambda$ as follows: One can find a fiber bundle
\[Z \stackrel{\pi}{\longrightarrow} X\]
and a function $f \in C^\infty(Z)$ such that
\begin{enumerate}
\item[(I)] the manifold $\Lambda_f$ (see (\ref{lambdaf})) intersects the horizontal subbundle
\[H^*Z :=\pi^* T^*X\]
transversally,
\item[(II)]  the canonical fiber bundle  map
\begin{equation}\label{HZtoTZ} H^*Z \to T^*X\end{equation}
maps the intersection $\Lambda_f \cap H^*Z$ diffeomorphically onto $\Lambda$.
\end{enumerate}
In fact we will take $Z= X\times\bbR^N$ to be the trivial bundle, in which case
	\[
	H^*Z = \left\{	(x,s,\zeta,\sigma)\;|\; \sigma=0 \right\}
	\]
	so that $\Lambda_f \cap H^*Z = C_f = \left\{ (x,s, (d_xf)_{(x,s)}, 0)\;;\;( d_sf)_{(x,s)} = 0	\right\}$.
Following H\"ormander we will call such an $f$ a \emph{generating function} for
$\Lambda \subset T^*X$ with respect to the fibration $\pi$.  The procedure of
passing from $\Lambda_f$ to $\Lambda$ is the {\em reduction} of $\Lambda_f$ with respect
to the co-isotropic submanifold $H^*Z$ of $T^*Z$.

Our adaption of this approach to the isotropic setting is as follows:
We will show (Proposition \ref{isotropicHormander}) that
if $\Sigma \subset T^*X$ is an isotropic submanifold, one can
(at least locally) find a Lagrangian submanifold $\Lambda \supset \Sigma$ in $T^*X$, a fibration $\pi: Z \to X$ and a function $f \in C^\infty(Z)$ with the properties (I) and (II) above, and a submanifold $Z_\Sigma \subset Z$ such that
\begin{enumerate}
\item[(III)] the manifold
\[\Sigma_f = \{(z, df_z)\ |\ z \in Z_\Sigma\}\] intersects the horizontal subbundle $H^*Z$ transversally, and
\item[(IV)]  the  bundle  map $H^*Z \to T^*X$
maps the intersection $\Sigma_f \cap H^*Z$ diffeomorphically onto $\Sigma$.
\end{enumerate}
Such a function $f$ will be called a non-degenerate phase function parametrizing the pair $(\Sigma, \Lambda)$.

Our generalization to the isotropic case (and in the semi-classical setting) of H\"ormander's characterization of $I^k(X, \Lambda)$ via generating functions is the following: We first equip $Z$ and $X$ with non-vanishing smooth measures, $\mu_Z$ and $\mu_X$, and define a push forward operation
\[
\pi_*: C_0^\infty(Z) \to C^\infty(X)
\]
with the defining property
\begin{equation}
\pi_*(g) \mu_X = \pi_* (g \mu_Z)
\end{equation}
for all $g \in C_0^\infty(Z)$, where $\pi_*$ on the right-hand side is the fiber integral operation. Then we claim
\begin{theorem}
\label{AltDefIsoSTh}
A function lies in $I^r(X, \Sigma)$ if and only if locally it is of the form
	\begin{equation}\label{AltDefIsoS}
	\pi_*(\Upsilon_0), \quad \Upsilon_0 \in I^{r-\frac N2}(Z, \Sigma_f),
	\end{equation}
where $f$ is a non-degenerate phase function parametrizing $(\Sigma, \Lambda)$ and $N=\dim Z - \dim X$ is the fiber dimension.
\end{theorem}

\begin{remark}
\mbox{\ }
\begin{enumerate}
\item As we will see this theorem translates in the integral representation (\ref{localExpression}).
\item In particular, if $\Sigma$
is a Lagrangian submanifold of $T^*X$, then these objects are semi-classical analogues of H\"ormander's {distribution of Lagrangian type},
and were discussed at length in \cite{GS}.
\end{enumerate}
\end{remark}

For $x$ a point in $X$ and $\Sigma=\{(x, \xi)\}$ a one-point set in $T^*X$ , the
associated isotropic functions have been around in the physics literature since the 1920's and are known
 as \emph{coherent states}.  There is a vast literature on coherent states both in mathematics and physics journals;
we mention work by R. Littlejohn, \cite{L}, G. Hagedorn, \cite{Ha}, and particularly
the monograph by M. Combescure and D. Robert \cite{CR} and references therein.

On coordinate patches centered at the point and if $\xi=df_0$, they have the simple form
      \[
      e^{if(x)/\hbar} \varphi(\hbar^{-1/2}x, \hbar),
      \]
      where $\varphi$ is a Schwartz function in $x$ that satisfies estimates of the form (\ref{2exp}), with $x''=x$.
      Thus as $\hbar$ tends to zero they become more and more concentrated at the point $x=0$ in configuration space and at the point $(0, \xi)$ in phase space.
In the last part of this paper we will  discuss isotropic functions
from a more intuitive perspective as ``superpositions
of coherent states".  In section 5 we will show
that every semi-classical oscillatory function of isotropic type can be
defined as a superposition of coherent states, and we discuss some consequences of this fact.

\medskip
For H\"ormander the main purpose of introducing  distributions of Lagrangian type is that it gave him a very clean simple
way of defining Fourier integral operators. Namely let $\Gamma: T^*Y \Rightarrow T^*X$ be a canonical relation,
and let $\Gamma^\# \subset T^*(X \times Y)$ be the associated Lagrangian submanifold, i.e.
\[
((x, \xi), (y, \eta)) \in \Gamma \Longleftrightarrow (x, y, \xi, -\eta) \in \Gamma^\#.
\]
Then the associated semi-classical Fourier integral operators $F: C_0^\infty(Y)\to C^\infty (X)$
are operators whose Schwartz kernels are elements in the space
$I^k(X \times Y, \Gamma^\#)$
for some $k$. This definition works equally well for \emph{isotropic} canonical relations, i.e.
relations $\Gamma: T^*Y \Rightarrow T^*X$ for which $\Gamma^\#$ is an isotropic submanifold.
Consequently, we will call the isotropic analogues of the operators above \emph{Fourier integral operators of isotropic type}. As a consequence of Theorem \ref{FIOinvThm}, one can easily show that if $F_1$ and $F_2$ are Fourier integral operators of isotropic type with microsupports on canonical relations $\Gamma_1$ and $\Gamma_2$ and these relations are cleanly composable, then $F_2 \circ F_1$ is a Fourier integral operator of isotropic type with microsupport on $\Gamma_2 \circ \Gamma_1$.
In \S 4 we prove that certain functions of elliptic operators are of this kind, which leads to
a variation of the Weyl eigenvalue counting function that appears to be new (Corollary \ref{cor:altWeyl}).
We would like to think that Mikhail Shubin would have enjoyed this result, if he did not know it already.

We hope to do a more systematic study of isotropic FIOs in the future.

%
%
%
%

\section{Local integral representations}

\subsection{Non-degenerate phase functions}

The purpose of this section is to prove  Theorem \ref{AltDefIsoSTh}, which gives an alternative definition of $I^r(X, \Sigma)$.
First we introduce some terminology:

\begin{definition}
	Let $X$ be a smooth manifold of dimension $n$ and $\Sigma \subset T^*X$ an isotropic submanifold. A \emph{framing} of $\Sigma$ is a Lagrangian submanifold $\Lambda \subset T^*X$ such that $\Sigma \subset \Lambda$.  The pair $(\Sigma, \Lambda)$ is called a {\em framed isotropic submanifold} of $T^*X$.
\end{definition}

The existence of framings for a given $\Sigma$ is not hard to establish;
the proof is sketched in Lemma 1  in  \cite{Gui}.
In what follows,
strictly speaking we will work with {germs} of framed isotropics, by which
we mean that we will not distinguish between $(\Sigma, \Lambda)$ and
$(\Sigma, \Lambda')$ if $\Lambda\cap\Lambda'$ contains a relative open set in each.
We will not use the language of germs explicitly, but germ-equivalence will be tacitly
assumed.

We have explained H\"ormander's extension of generating functions for arbitrary Lagrangian submanifolds in the introduction. The isotropic version of H\"ormander's theorem is as follows.

\begin{proposition}\label{isotropicHormander}
	Let $(\Sigma, \Lambda)$  be a framed isotropic submanifold of $T^*X$.
	Then there exists, at least for a neighborhood of each point in $\Sigma$,
	a fiber bundle $\pi: Z \to X$, a horizontal Lagrangian submanifold, $\Lambda_f \subset T^*Z$, with generating function $f \in C^\infty(Z)$, and an isotropic submanifold, $\Sigma_f$, of $\Lambda_f$ such that $\Lambda_f$ and $\Sigma_f$ intersect $H^*Z$ transversally and the projection map \eqref{HZtoTZ} maps $\Lambda_f \cap H^*Z$ onto $\Lambda$ and $\Sigma_f \cap H^*Z$ onto $\Sigma$.
	\end{proposition}

We now rephrase this proposition in the language of phase functions.
For any Lagrangian submanifold $\Lambda \subset T^*X$, recall that a function $f=f(x,s) \in C^\infty(X \times \mathbb R^N)$ is said to be a \emph{non-degenerate phase function parametrizing $\Lambda$} if
\begin{enumerate}
  \item zero is a regular value of the map
  \[X \times \mathbb R^N \to \mathbb R^N, \quad (x, s) \mapsto d_sf\] and
  \item the map
  \[\Phi: X \times \mathbb R^N\to T^*X, \qquad (x,s) \mapsto (x, (d_x f)_{(x,s)})\]
  is an embedding from the critical set
  \[
  C_f := \{(x,s)\;:\; d_s f = 0\}
  \]
  onto $\Lambda$ (it is automatically an immersion).
\end{enumerate}
Then Proposition \ref{isotropicHormander} can be restated as
\begin{proposition}\label{paramPhases}
Given a framed isotropic submanifold $(\Sigma,\Lambda)$,  there is a covering of
$\Lambda$ by relative
open sets $\Gamma\subset \Lambda$ such that, for each $\Gamma$, there exists:
\begin{enumerate}
  \item Open sets $\calV\subset X$ and $S\subset\bbR^N$,
  \item A splitting of the variables $S\ni s = (t, u)\in\bbR^K \times\bbR^l$,
  \item A function $f:\calV\times S\to\bbR$
\end{enumerate}
such that:
\begin{enumerate}
  \item[$(a)$] $f$ is a non-degenerate phase function parametrizing $\Gamma$,
  \item[$(b)$] The intersection
  \[
  C_f\cap \{u=0\}
  \]
  is transverse, and
  \item[$(c)$] Under the map $\Phi$ above, $C_f\cap \{u=0\}$ maps onto $\Sigma\cap\Gamma$.
\end{enumerate}
\end{proposition}
We postpone the proof until after Proposition \ref{nondegphaseforcomposition}.

\begin{definition}
A function $f$ satisfying (a), (b) and (c) above will be called a \emph{non-degenerate phase function parametrizing the framed isotropic} $(\Sigma, \Lambda)$.
\end{definition}

It is obvious that such parametrizations are not unique. But it is easy to check that, in any parametrization, the number $l$ of $u$ variables must be
the same, namely,
\[l = n - \dim(\Sigma).\]
Note that this $l$ coincides with the  notation used for the splitting of $\mathbb R^n$ in the introduction.

Now let $(\Sigma_1, \Lambda_1)$ be a framed isotropic submanifold of $T^*X$, $\phi: T^*X \to T^*Y$ be a symplectomorphism. Then $(\Sigma_2, \Lambda_2)=\phi(\Sigma_1, \Lambda_1)$ is a framed isotropic in $T^*Y$. Let
$\Gamma \subset T^*(X \times Y)$ be the Lagrangian submanifold associated with $\phi$, and let
\[g=g(x, y, s) \in C^\infty(X \times Y \times S)\]
be a non-degenerate phase function parametrizing $\Gamma$. Then by definition, the map
\[\Phi_2: X \times Y \times S \to T^*(X \times Y), \qquad (x,y,s) \mapsto (x,y, (d_{x,y} g)_{(x,y,s)})\]
maps the critical set
\[
C_g=\{(x, y, s)\ |\ d_sg=0\}
\]
diffeomorphically onto the $\Gamma$. As a consequence,
\begin{equation}\label{eqnrelgGamma}
  \phi(x, -(d_xg)_{x,y,s}) = (y, (d_yg)_{(x,y,s)}), \qquad \forall (x, y, s) \in C_g.
\end{equation}
In proving the main theorem of this section we will need

\begin{proposition}\label{nondegphaseforcomposition}
Let $f=f(x, t, u)$ be a non-degenerate phase function (with fiber variables $t$ and $u$) parametrizing the framed isotropic submanifold $(\Sigma_1, \Lambda_1)$ in $T^*X$, and let $g=g(y, x, s)$ be a non-degenerate phase function (with fiber variables $s$) parametrizing $\Gamma$, the Lagrangian submanifold associated with the symplectomorphism $\phi: T^*X \to T^*Y$. Then the function
\[F(y, x, s, t, u):= f(x, t ,u)+g(y, x, s)\]
is a non-degenerate phase function (with fiber variables $x, s, t$ and $u$) parametrizing the framed isotropic submanifold $(\Sigma_2, \Lambda_2)=\phi(\Sigma_1, \Lambda_1)$ in $T^*Y$.
\end{proposition}
\begin{proof}
Since $\phi$ is a symplectomorphism, the composition $\Gamma \circ \Lambda_1$ is transversal. As a consequence, $F$ is a non-degenerate phase function parametrizing the Lagrangian submanifold $\Lambda_2=\phi(\Lambda_1)$. In particular, the map
\[\Phi: Y \times (X \times S \times T \times U) \to T^*Y, \qquad (y, x, s, t, u) \mapsto (y, d_yF)\]
is a diffeomorphism from the critical set
      \[C_F:=\{(y, x, s, t, u): d_xF=0, d_tF=0, d_uF=0, d_sF=0\}\]
      onto $\Lambda_2$.
(For a proof of these two assertions, see, e.g. \S 4.3 and Theorem 5.6.1 in \cite{GS}).

It remains to prove
\begin{enumerate}
  \item[(a)] $C_F$ intersects $\{(y, x, s, t, u): u=0\}$ transversally and
  \item[(b)] the intersection $C_F \cap \{(y, x, s, t, u): u=0\}$ gets mapped onto $\Sigma_2$ under the map $\Phi$.
\end{enumerate}

Note that by definition,
\[
C_f=\{(x, t, u)\ :\ d_tf=0, d_uf=0\}
\]
intersects the set $\{(x, t, u)\ |\ u=0\}$ transversally and the intersection gets mapped diffeomorphically onto $\Sigma_2$ under the map
\[\Phi_1: X \times T \times U \to T^*X, \qquad (x, t, u)\mapsto (x, d_xf).\]

Denote $\pi: Y \times X \times T \times U \times S \to X \times T \times U$ be the standard projection map.
Then by definition, $\pi(C_F) \subset C_f$. We can show that $\pi|_{C_F}: C_F \to C_f$ is a diffeomorphism.
In fact, by definition, $(y, x, t, u, s) \in C_F$ implies $(x, y, s) \in C_g$ and $d_xf=-d_xg$.
In view of \eqref{eqnrelgGamma},
the following diagram commutes:
	\begin{equation}\label{CFCfdiamgram}
	\begin{array}{rcccc}
 Y \times X \times T \times U \times S \supset  & C_F  & \stackrel{\Phi}{\longrightarrow} & \Lambda_2 & \subset T^*Y\\
	& {\small \pi}{\downarrow} & & {\uparrow\phi}\\
	X \times T \times U \supset & C_f  & \stackrel{\Phi_1}{\longrightarrow} & \Lambda_1 & \subset T^*X
	\end{array}.
	\end{equation}
So $\pi$ maps $C_F$ diffeomorphically onto $C_f$. Since $C_f$ intersects $\{(x, t, u)\ |\ u=0\}$ transversally, we conclude that $C_F$ intersects $\{(y, x, t, s, u): u=0\}$ transversally.

Finally, to prove the assertion (b), it is enough to chase the commutative diagram \eqref{CFCfdiamgram} with $C_F$ replaced by $C_F \cap \{(y, x, t, s, u): u=0\}$, to get
	\begin{equation}\label{CFCfSigmadiamgram}
	\begin{array}{ccc}
  C_F  \cap \{(y, x, s, t, u): u=0\} & \stackrel{\Phi}{\longrightarrow} & \Sigma_2 \\
	{\small \pi}{\downarrow} & & {\uparrow\phi}\\
	 C_f  \cap \{(x, t, u): u=0\}  & \stackrel{\Phi_1}{\longrightarrow} & \Sigma_1
	\end{array}.
	\end{equation}

\end{proof}

\begin{proof}
[{Proof of Proposition \ref{paramPhases}:}]  It is easy to see that, locally, any
framed isotropic $(\Lambda, \Sigma)$ can be mapped by a canonical transformation to the
pair $(\Lambda_0,\Sigma_0)$ where $\Lambda_0\subset T^*\bbR^n$ is the zero section and $\Sigma_0$ the
model isotropic (\ref{Sigma0}) of the appropriate dimension.  By Proposition \ref{nondegphaseforcomposition},
it suffices to show that there is a non-degenerate phase function parametrizing $(\Lambda_0,\Sigma_0)$,
and one readily can check that
\begin{equation}\label{}
	f(x',x'', t,u ) = t\cdot (x''-u)
\end{equation}
is such a phase function.
\end{proof}

\subsection{Amplitudes and the local form}
We now introduce smooth amplitudes
\[
a(x, t, u, \hbar)\in C^\infty(\calV\times \calT\times \bbR^l \times (0, \hbar_0)),\quad \calV\subset\bbR^n,\ \calT\subset\bbR^K\ \text{open},
\]
which are Schwartz functions in the $u$ variable, with estimates locally uniform in
$(x,t)$.  We will also have to allow for $\h$-dependence, in the $C^\infty$ topology
in $(x,t)$:
\begin{equation}\label{amplitutea}
\forall \alpha, \, \beta,\,\gamma,\, m \quad\exists C\ \text{such that} \ \forall(x, t, u)\in X\times \bbR^N
\quad |\partial^\alpha_x\partial^\beta_t\partial_u^\gamma a(x,t,u,\hbar)| \leq C \left(1+\norm{u}\right)^{-m}.
\end{equation}

We will further assume that there exists a sequence $\{a_j(x,t,u)\}$ of smooth functions, compactly
 supported in the $t$ variable and Schwartz  in the $u$ variable, and such that
\begin{equation}\label{amplit}
a(x,t,u,\h) \sim \sum_{j=0}^\infty a_j(x,t,u)\, h^{j/2}.
\end{equation}

We can now restate Theorem \ref{AltDefIsoSTh} as follows, which is the main result of this section:

\begin{theorem}\label{IntegralRep}
Given a parametrization of $(\Sigma,\Lambda)$, $\Lambda\subset T^*X$,
by a phase function
$f: X\times \calT\times U\to\bbR$ where $U\subset \bbR^l$, $\calT \in \bbR^K$ and $N=K+l$, then for amplitudes given by \eqref{amplit},
\begin{equation}\label{localExpression}
\Upsilon(x,\h) = \h^{r-\frac N2}\int_{\calT\times U} e^{\sqrt{-1}\h^{-1}f(x,t,u)} a(x,t,\hbar^{-1/2}u, \h)\, dt du
\end{equation}
is in the class $I^r(X,\Sigma)$.  Conversely, if $\Upsilon\in I^r(X, \Sigma)$ and $\sigma_0\in\Sigma$, then microlocally
near $\sigma_0$, $\Upsilon$ is equal to an integral of the previous form.
\end{theorem}
Note that the integration over $u$ is only over the open set $U$.   However, as a function of $u\in U$,
the function $a$ is defined on all of $\bbR^l$, which in particular implies that for any $\h$ the integrand
above is well-defined and the integral converges absolutely.
Note also that $a$ may be chosen to be compactly supported in the $u$ variable.

\bigskip
A preliminary result is:
\begin{lemma}\label{WFSet}
If $\Upsilon$ is as in \eqref{localExpression}, then its semiclassical wave-front set equals
\begin{equation}\label{wfset}
\text{WF}(\Upsilon) = \{ (x,(d_xf)_{(x,t,0)})\;;\; (d_sf)_{(x,t,0)} = 0\ \text{and}\ (x,t,0)\in\mathrm{supp}(a)\},
\end{equation}
where we have let $s=(t,u)$.
In particular
\begin{equation}\label{wfset2}
\text{WF}(\Upsilon) \subset \Sigma.
\end{equation}
\end{lemma}
\begin{proof}
Let us begin by considering a function of the form
\begin{equation}\label{}
\Upsilon_0(x,u) = e^{i\h^{-1}f(x,u)}\, a(x,h^{-1}u)
\end{equation}
on an open set of the form $X\times U$.  (We will see below that in fact $\Upsilon_0$ is one
of the oscillatory functions associated to the pair $(\Sigma,\Lambda)$ where $\Lambda$ is the
graph of $df$ and $\Sigma$ the part of $\Lambda$ where $u=0$.)
It is clear that, for any compact set $D\subset X\times U$ such that $D\cap \left(X\times\{0\}\right) = \emptyset$, one has
\[\Upsilon_0|_D = O(\hbar^\infty)\]
in the $C^\infty$ topology, by the fact that $a(x,u)$ is Schwartz in $u$.  On the other hand
the wave-front set of $e^{i\h^{-1}f(x,u)}$ is just the graph of $df$.  From this it follows that
\begin{equation}\label{}
\text{WF}(\Upsilon_0) = \{ (x, u ; df_{(x,u)})\;:\; u=0\ \text{and}\ (x,0)\in\text{supp}(a)\}.
\end{equation}

Consider next a function of the form
\begin{equation}\label{2ndApprox}
\Upsilon_1(x) = \int_U e^{i\h^{-1}f(x,u)}\, a(x,h^{-1}u)\, du.
\end{equation}
This is simply the push-forward by the natural projection $X\times U\to X$ of a function $\Upsilon_0$
of the form considered above.  By the calculus of wave-front sets, we obtain for the wave-front set of
(\ref{2ndApprox})
\begin{equation}\label{wf2ndApprox}
\text{WF}(\Upsilon_1) =
\{ (x, (d_xf)_{(x,u)})\;:\; (d_uf)_{(x,u)}=0\ \text{and}\ u=0\ \text{and}\ (x,0)\in\text{supp}(a)\}.
\end{equation}

Consider now the general case (as in the statement of the Lemma).  By (\ref{wf2ndApprox})
we know the wave-front set of
\[
\Upsilon_2(x,t) =  \int_{ U} e^{i\h^{-1}f(x,t,u)} a(x,t,\hbar^{-1/2}u, \h)\, du, \quad (x,t)\in U\times\calT.
\]
But $\Upsilon = \pi_*(\Upsilon_2)$, where $\pi: X\times\calT\to X$ is the natural projection.
The general result follows, again by the calculus of wave-front sets.
\end{proof}

\subsection{Proof of the first half of Theorem \ref{IntegralRep}}

To prove the first half of Theorem \ref{IntegralRep}, we first prove a special case:
\begin{lemma}\label{specialUpsilon}
Let $x=(x',x'')$ denote the variables in $\bbR^n=\bbR^{k+l}$, $\Lambda_0\subset T^*\bbR^{k+l}$
the zero section, and $\Sigma_0 = \Lambda_0\cap\{x''=0\}$.  Let $\calV\subset \bbR^{k+l}$,
$\mathcal S \subset\bbR^{K+l}$
be open sets and $f:\calV\times\mathcal S\to\bbR$, $f=f(x,t,u)$
be a non-degenerate phase function parametrizing $(\Sigma_0, \Lambda_0)$.
Let $a(x,t,u, \hbar)$ be an amplitude as in (\ref{amplit}).  Then
\[
\Upsilon(x, \hbar) = \h^{r-\frac N2}\int e^{\sqrt{-1}\hinv f(x,t,u)} a(x,t, u/\sqrt{\h}, \hbar)\, dt\, du
\]
is in the model class of isotropic functions $I^r(\mathbb R^n, \Sigma_0)$.
\end{lemma}
\begin{proof}
We will apply the method of stationary phase to the integral.  For this we need to
understand the critical points of the phase with respect to the fiber variables $s = (t,u)$.

By definition, for the critical set $C_f=\{(x, s) \ |\ d_s f=0\}$  we know that
\begin{itemize}
  \item the jacobian $\mathbb J = \begin{pmatrix} \frac{\partial^2 f}{\partial x \partial s} & \frac{\partial^2 f}{\partial s\partial s} \end{pmatrix}$
  has full rank at each point on $C_f$,
  \item the projection $\pi_1: C_f \to \calV, (x, s) \mapsto x$ is a local diffeomorphism,
  \item  $C_f$ intersect transversally with the set $\{u=0\}$,
  \item the intersection $C_f \cap \{u=0\}$ is diffeomorphic to $\Sigma_0$.
\end{itemize}

Take any point $p$ in the intersection $C_f \cap \{u=0\}$. Then as we know, the map
\[\pi_1: C_f \to \calV, \quad (x',x'',t,u) \mapsto (x',x'')\]
is a local diffeomorphism near $p$. So in particular, there exist  functions $f_1, f_2$ so that on $C_f$, $t=f_1(x', x'')$ and $u=f_2(x',x'')$ near $p$.   We also know that $C_f$ intersects $\{u=0\}$ transversally and the map
\[\Psi: C_f \cap \{u=0\} \to \Sigma_0, \quad (x',x'',t, u) \mapsto (x',x'',t,0) \mapsto (x',0,0,0)\]
is a local diffeomorphism near $p$. As a consequence, we see that the matrix $(\frac{\partial f_2}{\partial x''})$ is non-degenerate.
Therefore, there is a smooth function $g$ so that on $C_f$,  $x''=g(x',u)$ near $p$. Now we define a new coordinate system of $\mathbb R^{n+N}$ near $p$,
\[
(x', x'', t, u_{\text{\em new}}) = (x', x'', t, g(x',u)),
\]
and we denote the inverse coordinate transformation by
\[
(x', x'', t, u) = (x', x'', t, \tilde g(x', u_{\text{\em new}})).
\]
Then under this new coordinate system, the critical set $C_f$ is contained in the set $\{x''=u_{\text{\em new}}\}$.
As a consequence, the phase function $f$ can be written, in these new coordinates, as
\[
f(x',x'',t,u) = f_0(x',t,u)  +\sum_{|\alpha|=2} f_\alpha(x',x'', t, u)(x''-u)^\alpha.
\]
According to the condition that $f$ parametrizes $(\Sigma_0, \Lambda_0)$, we conclude that

\[f_0(x',t,u) = C_0+\sum_{|\gamma|=2}h_\gamma(x', t)t^\alpha + \sum_{|\beta|=2}g_\beta(x', t, u) u^\beta.\]
It follows that 
\[\aligned
\Upsilon(x)= e^{\sqrt{-1}\hinv C_0}\int & e^{\sqrt{-1} \hinv\sum h_\gamma(x', t)t^\alpha + g_\beta(x', t, u)u^\beta+ f_\alpha(x',x'', t, u)(x''-u)^\alpha} \\
  & \cdot \tilde a(x',x'',t, u, \hbar)\det J\, dt\, du,
\endaligned\]
where the summation is over ${|\alpha|=2, |\beta|=2, |\gamma|=2}$,
\[
\tilde a(x',x'', t, u, \hbar) = a(x',x'', t, \tilde g(x', u)/\sqrt \hbar, \hbar),
\]
and $J=(\frac{\partial g}{\partial u})$. We want to show that the function $\Upsilon(x', \sqrt{\hbar}x'', \hbar)$ is Schwartz in $x''$. In fact, by changing variables $t \to \sqrt \hbar t$ and $u \to \sqrt \hbar u$, we have
\[
\aligned
\Upsilon(x', \sqrt{\hbar}x'', \hbar )= \hbar^{N/2} e^{\sqrt{-1}\hbar^{-1}C_0}\int &  e^{\sqrt{-1} \sum h_\gamma(x', \sqrt \hbar t)t^\alpha + g_\beta(x', \sqrt \hbar t, \sqrt \hbar u)u^\beta+ f_\alpha(x', \sqrt \hbar  x'', \sqrt \hbar t, \sqrt \hbar u)(x''-u)^\alpha} \\
& \tilde a(x', \sqrt \hbar x'', \sqrt\hbar t, \sqrt{\hbar}u, \hbar) \det J dtdu.
\endaligned
\]
Noting that $\tilde g(x', \sqrt \hbar u)/\sqrt h$ has an expansion in nonnegative powers of $\hbar$, we see that
\[
\tilde  a(x', \sqrt \hbar x'', \sqrt\hbar t, \sqrt{\hbar}u, \hbar) \sim \hbar^{r-\frac N2} \sum \tilde a_j(x, t, u)\hbar^{j/2},
\]
where each $\tilde a_j$ is Schwartz in $u$. Now the conclusion follows.
\end{proof}

To prove the general case, we let $\phi: T^*X \to T^*\mathbb R^n$ be a symplectomorphism such that
\[\phi(\Sigma, \Lambda)=(\Sigma_0, \Lambda_0).\]
Let $F$ be a semi-classical Fourier integral operator of degree zero quantizing $\phi$
(We are following the definition of order of an FIO of \cite[\S 8.2 ]{GS}.). Then by definition, there exists an amplitude $b$ such that
\begin{equation}\label{}
F(\Upsilon) (y) = \hbar^{r-(N+n+N')/2}\int e^{\sqrt{-1}\hinv (g(x,y,s)+f(x, t, u))}\ b(x,y,s,\h)\, a(x, t, u/\sqrt \hbar, \hbar)\, dx ds dt du,
\end{equation}
where $N'$ is the number of $s$-parameters.
According to Proposition \ref{nondegphaseforcomposition}, the new phase function $g(x,y,s)+f(x, t, u)$ is a non-degenerate phase function (with parameters $x, s, t, u$) parametrizing $(\Sigma_0, \Lambda_0)$.  According to Lemma \ref{specialUpsilon}, $F(\Upsilon)$ is an element of $I^r(\mathbb R^n, \Sigma_0)$.  It follows that $\Upsilon \in I^r(X, \Sigma)$.

\subsection{Applying an FIO associated to a canonical transformation to the model case}

Finally we prove the second part of Theorem \ref{IntegralRep}. As before we split $\mathbb R^n=\mathbb R^k \times \mathbb R^l$, let $\calU,\,\calV\subset \bbR^{k+l}$ be open sets,
and $F: C_0^\infty(\calV)\to C^\infty(\mathcal U\times (0,\h_0))$
a semi-classical FIO of degree zero 
associated to a canonical transformation  $\Phi:\mathcal U\times\bbR^n\to \calV\times\bbR^n$. Let $f(x,y,s)\in C^\infty(\mathcal U\times\calV\times\mathcal S)$, $\mathcal S \subset\bbR^N$,
be a non-degenerate phase function parametrizing the graph of the canonical
transformation $\Phi$.
Let $\Upsilon$ be a function in our model space, (\ref{basicdef}). The result of applying $F$ to a model isotropic state is an integral of the form
\begin{equation}\label{fupsilon}
F(\Upsilon) (x) = \hbar^{-(n+N)/2}\int e^{\sqrt{-1}\hinv f(x,y',y'',s)}\ b(x,y',y'',s,\h)\, \varphi(y', y''/\sqrt{\h})\, dy'\, dy''\, ds.
\end{equation}

 We will show that $F(\Upsilon)$ is locally of the form (\ref{localExpression}).

\begin{lemma}
Let $f(x,y,s)\in C^\infty(\mathcal U\times\calV\times\mathcal S)$, $\mathcal S \subset\bbR^N$,
be a non-degenerate phase function parametrizing the graph of a canonical
transformation $\Phi:\mathcal U\times\bbR^n\to \calV\times\bbR^n$.
Then $f$, considered as a phase function with fiber variables $(y,s)$, is
a non-degenerate phase function parametrizing the image of the zero
section $\mathcal U\times\{0\}$ under the canonical transformation.
Moreover, if we split $y=(y',y'')$, then $f(x,y',y'',s)$ (where the fiber variables
are $(y',y'',s)$) parametrizes the isotropic $\Phi(\Sigma_0)$.
\end{lemma}
\begin{proof}
The assumption on $f$ is that zero is a regular value of the map
$(x,y,s)\mapsto d_s f(x,y,s)$, and that the map
\begin{equation}\label{zreg}
C_f:=\{(x,y,s)\;;\; d_s f(x,y,s)=0\}\longrightarrow (y, -d_y f\;;\; x, d_x f)
\end{equation}
is a (local) diffeomorphism onto the graph of $\Phi$.  It follows that the map
\[
C_f \longrightarrow d_y f
\]
is a submersion.  From this we will argue that zero is a regular value of the map
\[
(x,y,s)\mapsto \left(d_y f(x,y,s)\,,\,d_s f(x,y,s)\right),
\]
which means that $f$ is a non-degenerate phase function in the fiber variables $(y,s)$.

Indeed the assumption is that the jacobian
\[
{\mathbb J} = \begin{pmatrix}
\frac{\partial^2 f}{\partial x\partial s} & \frac{\partial^2 f}{\partial y\partial s} &
\frac{\partial^2 f}{\partial s\partial s}
\end{pmatrix}
\]
is surjective at each point of $C_f$, and we also know that the jacobian
\[
{\mathbb J}_y =
\begin{pmatrix}
\frac{\partial^2 f}{\partial x\partial y} & \frac{\partial^2 f}{\partial y\partial y} &
\frac{\partial^2 f}{\partial s\partial y}
\end{pmatrix}
\]
(differential of the map
$(x,y,s)\mapsto d_yf (x,y,s)$) is surjective when restricted to the kernel of ${\mathbb J}$,
at each point of $\wt C_f:= C_f\cap \{ d_y f=0\}$.
From this it follows that the jacobian
\[
\wt{\mathbb J} = \begin{pmatrix}
 & {\mathbb J}_y & \\
 & {\mathbb J} &
\end{pmatrix}
\]
is surjective when restricted to $\wt C_f$.

It is clear from (\ref{zreg}) that
\[
\wt C_f := C_f\cap \{ d_y f=0\}\mapsto (x, d_x f)
\]
maps $\wt C_f$ locally diffeomorphically onto the image of the zero section under $\Phi$.

\smallskip
The proof of the second assertion in the lemma is proved in a similar way,
using that the map $\wt C_f\ni (x,y',y'',s)\mapsto u$ is a submersion.
\end{proof}

\bigskip
Let us go back to (\ref{fupsilon}).  By the previous lemma the phase is a non-degenerate
phase function parametrizing the image of $\Sigma_0$ under the underlying canonical transformation, and we have the correct power of $\hbar$ since the fiber variable has dimension $n+N$ now.
It remains to show that the amplitude
\[
b(x,y',y'',s,\h)\, \varphi(y', y''/\sqrt{\h})
\]
is of the appropriate kind.  This follows by Taylor expanding the amplitude $b$ in the $y''$ variables, at $y''=0$:
\[
b(x,y',y'',s,\h) \sim \sum_\alpha (y'')^\alpha b_\alpha(x,y',s,\h).
\]
It then follows that
\[
b(x,y',y'',s,\h)\, \varphi(y', y''/\sqrt{\h}) \sim \sum_\alpha  b_\alpha(x,y',s,\h)\,
\h^{|\alpha|/2}\, \ ({y''}/{\sqrt{\h}})^\alpha\, \varphi(y', y''/\sqrt{\h}),
\]
and, since the functions $(y'')^\alpha \varphi(y',y'')$ are Schwartz in $y''$, we can conclude that
(\ref{fupsilon}) is an integral of the desired form.

\subsection{Global existence and Bohr-Sommerfeld conditions}

In the previous discussion (and in fact in our previous paper \cite{GUW}), we have ignored the issue of
patching together local isotropic functions to obtain a global one.
Here we discuss how. In order to have a global isotropic function, we need
that the isotropic $\Sigma$ satisfies a Bohr-Sommerfeld condition, which is a condition on the de Rham cohomology
class of $\iota_\Sigma^*(pdx)$ where $\iota_\Sigma: \Sigma\hookrightarrow T^*X$ is the inclusion and $pdx$ the
canonical one-form of $T^*X$.

We begin by reviewing the situation in the Lagrangian case.
If $f$ is a non-degenerate phase function parametrizing  a Lagrangian submanifold $\Lambda\subset T^*X$, so is
$f+c$ for any constant $c\in\bbR$.  Adding a constant to the phase has the effect of multiplying an associated
Lagrangian state, written using the phase $f$, by $e^{-i\hinv c}$ (which does not change it as a quantum state).
However, this additional phase factor can be
important to keep track of in applications.   To remedy this ambiguity, we proceed as in \cite{GS} Chapter 12 (an equivalent
approach also appeared in \cite{PU}).

Let us begin by noticing that if $f\in C^\infty(U\times\bbR^N)$ is a non-degenerate phase function such that
$C_f\to\Lambda$ is an embedding onto a relative open set $\Gamma\subset\Lambda$, then
we can form the composition
\begin{equation}\label{above}
\psi_f: \Gamma\to C_f\hookrightarrow U\times\bbR^N\xrightarrow{f}\bbR.
\end{equation}
\begin{lemma} If $\iota_\Gamma:\Gamma\subset T^*X$ is the inclusion, then  $d\psi_f = \iota_\Gamma^*(pdx)$,
where $pdx$ is the canonical one-form of $T^*X$.
\end{lemma}
\begin{proof}
Consider $C_f\to\Gamma\hookrightarrow T^*X$.  The pull-back of $pdx$ by this map is the same as
the pull-back of $df$ by the inclusion $C_f\subset U\times\bbR^N$, because by definition the partial
derivatives of $f$ with respect to the fiber variables vanish on $C_f$.
\end{proof}
It follows from this that given two phase functions $f,f'$ parametrizing relative open sets $\Gamma,\Gamma'\subset\Lambda$
one must have that $(\psi_f-\psi_{f'})|_{\Gamma\cap\Gamma'}$ is locally constant.  To fix these constants globally, one needs  that
$\Lambda$ satisfies a Bohr-Sommerfeld condition.  The strongest one is the following:
\begin{definition}
A Lagrangian submanifold $\Lambda\subset T^*X$ is called exact (or strongly admissible, in the terminology of
\cite{PU}) iff there exists $\psi\in C^\infty(X)$ such that $d\psi = \iota_\Lambda^*(pdx)$ (that is, if
the de Rham cohomology class of $\iota_\Lambda^*(pdx)$ is zero).
\end{definition}

For such manifolds, one can fix $\psi$ and use local integral expressions with phase functions $f$ such that
$\psi_f = \psi|_\Gamma$, using the notation (\ref{above}).  Using such phase functions one can glue locally-defined
Lagrangian states, using a microlocal partition of unity.

The condition of being exact is too restrictive for some applications.  One can relax the condition a bit, as follows:

\begin{definition}
$\Lambda$ satisfies the Bohr-Sommerfeld condition iff the cohomology class of $\iota_\Lambda^*(pdx)$
is integral.  Equivalently, iff there exists a function $\chi = e^{2\pi i\psi}: \Lambda\to S^1$ such that
\[
\frac{1}{2\pi i} \frac{d\chi}{\chi} = \iota_\Lambda^*(pdx).
\]
\end{definition}
Cohomologically, the condition is that $\iota_\Lambda^*(pdx)$ represents an integral cohomology class of $\Lambda$.
Under this condition the expression $e^{i\hinv \psi}$ is a well-defined function on $\Lambda$ provided
$\h = 1/k$ for some $k\in\bbN$.

As we have seen, given an isotropic $\Sigma$, to parametrize it we need to frame it first, and the
previous considerations apply to any framing of $\Sigma$.  Since it is enough to consider framings that lie in a small neighborhood
of $\Sigma$,we can always assume that we have framed $\Sigma$ by a Lagrangian $\Lambda$ such that $\Sigma$ is a deformation
retract of $\Lambda$.  Therefore, the Bohr-Sommerfeld condition is that  $\iota_\Sigma^*(pdx)$ represents an integral cohomology class of $\Sigma$.

\section{General FIO invariance}

In this section we will prove the general FIO invariance of the class of isotropic functions under the standard clean composition condition
(Theorem \ref{FIOinvThm}).

\subsection{Geometric considerations}
Let $\Sigma\subset T^*Y$ be an isotropic submanifold.
We suppose $F: C_0^\infty(Y) \to C^\infty(X)$ is a semi-classical FIO whose canonical relation $\Gamma: T^*Y \Rightarrow T^*X$ composes with $\Sigma$ cleanly. Recall that this means that
\begin{equation*}\label{CI}
\tag{C.I.}\Gamma \ \ \text{intersects}\ \  T^*X\times\Sigma  \ \ \text{cleanly.}
\end{equation*}
If this is the case, the {\em excess} of the intersection is the number
\begin{equation}\label{eq:excess}
	e = \dim \left[\Gamma\cap \left(T^*X\times\Sigma\right)\right] + \dim (T^*X\times T^*Y )-
	\left(\dim\Gamma + \dim(T^*X\times\Sigma)\right),
\end{equation}
which is zero precisely when the intersection is transverse.  As a corollary of the clean intersection
condition, the set
\[
\Gamma(\Sigma) := \{p \in T^*X\ |\ \exists q \in \Sigma \text{\ such that\ }(p, q) \in \Gamma\}
\]
is an immersed isotropic submanifold of $T^*X$, see \cite[Proposition 7.1]{BG}.

Since the classes of isotropic functions are invariant under the action of FIOs associated to canonical transformations,
without loss of generality we can assume that $\Sigma$ is contained in the zero section of $T^*Y$.
Moreover, since the statement is micro-local, we will be making various simplifying assumptions along the course of the proof.
In particular we can assume that  $Y = Y'\times Y''$ where $Y''\subset\bbR^l$ and $Y'\subset\bbR^{\dim Y-l}$ are open sets,
$0\in Y''$, and that
\[
\Sigma = Y'\times \{0\},
\]
where we are identifying $Y'$ with the zero section of $T^*Y'$.
We will denote the coordinates on $Y$ by
\[
y = (y', y''),\quad y'= (y'_1,\ldots, y'_{\dim Y-l}), \quad y'' = (y''_1,\dots , y''_l).
\]
If $\eta = (\eta', \eta'')$ are the dual coordinates, then $\Sigma$ is defined by the equations:
\begin{equation}\label{eqnsSigma}
\Sigma:\qquad y''=0, \quad \eta =0.
\end{equation}

\begin{lemma}\label{compatibleGenFctn}
Let $\Gamma\subset T^*X\times T^*Y$ be a Lagrangian submanifold.  Then there exists a non-degenerate
generating function $f\in C^\infty(X\times Y\times S)$ of $\Gamma$, where $S\subset\bbR^N$ is open, such that the differentials
\begin{equation}\label{differentialsL.I.}
dy''_1, \cdots ,dy''_l,\ d\left(\frac{\partial f}{\partial y''_1}\right), \ldots , d\left(\frac{\partial f}{\partial y''_l}\right)
\end{equation}
are linearly independent everywhere.
\end{lemma}
\begin{proof}
	Start with an arbitrary non-degenerate generating function $f_0\in C^\infty(X\times Y\times S_0)$.
	Let $S := S_0\times \bbR^l$, and define $f_1\in C^\infty(X\times Y\times S)$ by
	\[
	f_1(x, y', y'', s_0, s'') = f_0(x, y', y'', s_0) + \frac 12 |s''|^2
	\]
	where $s_0\in S_0$ and $s''\in\bbR^l$.  Now let $g:S\to S$ be the diffeomorphism
	\[
	g(x,y',y'',s_0, s'') = (x,y', y'', s_0, y''+s'')
	\]
	and define
	\[
	f(x,y', y'', s_0, s'') = (f_1\circ g) (x,y', y'', s_0, s'') = f_0(x,y',y'',s_0) + \frac 12 \left| y''+s''\right|^2.
	\]
	The three functions $f_0, f_1, f$ are equivalent: they are non-degenerate and parametrize $\Gamma$.
	Moreover
	\[
	d\left(\frac{\partial f}{\partial y''_i}\right) = d\left(\frac{\partial f_0}{\partial y''_i}\right)  + dy''_i+ds''_i,
	\quad 1\leq i\leq l.
	\]
	The differentials
	\[
	dy_1'', \cdots, dy_l'', d\left(\frac{\partial f}{\partial y''_1}\right), \cdots, d\left(\frac{\partial f}{\partial y''_l}\right)
	\]
	are linearly independent everywhere, because of the presence of the $ds''_i$ in the second set
		of differentials. This completes the proof.
\end{proof}

\medskip
For the rest of this section we will assume that $f$ is an {\em adapted phase function}, meaning, as in this Lemma.  We let,
as usual,
\[
\calC_f := \left\{(x,y,s)\in X\times Y\times S\;|\; \frac{\partial f}{\partial s}(x,y,s)=0\right\}
\]
and $\Phi: \calC_f\to T^*X\times T^*Y$ be the map
\[
\Phi(x,y,s) = (x,y,  \partial_xf(x,y,s), \partial_yf(x,y,s))
\]
which we assume is a diffeomorphism onto $\Gamma$.  Finally,
we assume that (\ref{CI}) holds.

\begin{lemma}
Under the previous assumptions, at each point $\kappa\in \Gamma\cap\left(T^*X\times\Sigma\right)$,
$\kappa = \Phi(x,y,s)$, the space
\begin{equation}\label{theVSpace}
d\Phi^{-1}_{(x,y,s)}T_\kappa (\Gamma\cap\left(T^*X\times\Sigma\right))\subset \bbR^{\dim X}\times\bbR^{\dim Y}\times \bbR^N
\end{equation}
is the solution set of the system of equations
\begin{equation}\label{sysEqns}
d\left(\frac{\partial f}{\partial s_i}\right) =0,\quad dy''_j = 0, \quad d\left(\frac{\partial f}{\partial y''_j}\right) =0,
\quad d\left(\frac{\partial f}{\partial y'_a}\right) =0
\end{equation}
for $1\leq i\leq N$, $1\leq j\leq l$, $1\leq a\leq \dim Y-l$.
\end{lemma}
\begin{proof}
	By the clean intersection condition
	\[
	T_\kappa (\Gamma\cap\left(T^*X\times\Sigma\right)) = T_\kappa \Gamma\cap T_\kappa \left(T^*X\times\Sigma\right).
	\]
	Thus the space (\ref{theVSpace}) consists of the vectors in $T_{(x, y, s)} \calC_f$ that are mapped into $T_\kappa \left(T^*X\times\Sigma\right)$
	by $d\Phi$.  The first set of equations (\ref{sysEqns}) carves out $T_{(x,y,s)}\calC_f$, since $f$ is non-degenerate, and the remaining
	equations correspond to the condition that the image is tangent to the intersection.
\end{proof}

Next we extract a {\em maximal} set of linearly independent differentials among those appearing in (\ref{sysEqns}).  We will
include the $2l$ differentials corresponding to the variables $y''_j$.  It will be convenient to re-label the coordinates on $Y'\times S$
as follows:

{\em  We will write $t = (t', t'')$ for a re-labeling of the variables $(y'_1,\ldots, y'_{\dim Y-l}, s_1,\ldots, s_N)$ where:
	\begin{enumerate}
\item 	$t' = (t'_1, \ldots , t'_m)$ is such that  the differentials
\[
\quad dy''_j , \quad d\left(\frac{\partial f}{\partial y''_j}\right) ,\quad d\left(\frac{\partial f}{\partial t'_b}\right)
\]
$1\leq j\leq l$, $1\leq b\leq m$ are linearly independent near $(x,y,s)$, and
\item $m$ is maximal with respect to property (1).
	\end{enumerate}
}
(Since the intersection is a manifold it is clear that such re-labeling is possible.)
The number $m$ has the following interpretation:
\begin{corollary}
The dimension of the intersection $\Gamma\cap \left(T^*X\times\Sigma\right)$ is
\begin{equation}\label{theDimInter}
\dim \Gamma\cap \left(T^*X\times\Sigma\right) = \dim X + \dim Y + N - (2l+m)
\end{equation}
and the excess is
\begin{equation}\label{theExcess}
e = N + \dim Y - (l+m).
\end{equation}
\end{corollary}
\begin{proof}
	The dimension formula follows from the fact that the codimension of $\Phi^{-1}(\Gamma\cap \left(T^*X\times\Sigma\right))$ in
	$X\times Y \times S$ is $2l+m$.  The excess formula follows from this and (\ref{eq:excess}):
	\[\aligned
	e &= \dim X + \dim Y + N - (2l+m) + 2\dim X + 2\dim Y
		-\left(\dim X + \dim Y  + 2\dim X +\dim Y - l
	\right) \\
&= N + \dim Y - (l+m).
	\endaligned\]
\end{proof}

\subsection{End of the proof}

It suffices to prove the Theorem for an FIO of order zero, which we represent by an oscillatory integral with
an adapted phase function.  Therefore
\begin{equation}\label{FUpsilonRevised}
F(\Upsilon) (x) = \hbar^{r-(\dim X+ N)/2}\int_{Y' \times Y'' \times \mathcal S} e^{\sqrt{-1}\hbar^{-1}f(x, y', y'', s)}a(x, y',\hbar^{-1/2}y'', s, \hbar) dy'dy''ds.
\end{equation}
Let $t=(t',t'')\in \calT' \times \mathcal T''$ as above.
By viewing $t''$ as parameters, we can write  the integral (\ref{FUpsilonRevised}) as
\begin{equation}\label{IteratedIntegral}
\h^{r-(\dim X+N)/2}\int_{\mathcal T''}\left(\int_{\mathcal T' \times Y''} e^{\sqrt{-1}h^{-1}f(x,t',y''; t'')}a(x,t',h^{-1/2}y'',\hbar; t'')dt'dy''\right)dt''.
\end{equation}
Note that for each fixed $t''$, the phase function
\[
\tilde f_{t''}(x, t', y'') := f(x,t',y'', t'')
\]
(with fiber variables $t',y''$)
is a non-degenerate phase function parametrizing the isotropic submanifold
\[
\Sigma_{t''}  = \left\{\left(x, \frac{\partial \tilde f}{\partial x}(x,t',0)\right)\ |\   \frac{\partial \tilde f_{t''}}{\partial t'}(x, t', 0)=0, \frac{\partial \tilde f_{t''}}{\partial y''}(x,t',0)=0\right\}
\]
with framing
\[
\Gamma_{t''} = \left\{\left(x, \frac{\partial \tilde f}{\partial x}(x,t',y'')\right)\ |\ \frac{\partial \tilde f_{t''}}{\partial t'}(x, t', y'')=0, \frac{\partial \tilde f_{t''}}{\partial z''}(x,t',z'')=0\right\}.
\]
Moreover,
\[
\left(x, \frac{\partial \tilde f_{t''}}{\partial x}(x, t', y'')\right) = \left(x, \frac{\partial f}{\partial x}(x, t', t'', y'')\right).
\]
So all these $\Sigma_{t''}$ coincide within an open set inside $\Gamma(\Sigma)$, i.e. they are the same in the sense of germs.
By Theorem \ref{IntegralRep}, for each  $t''$ the
integrand in (\ref{IteratedIntegral}) is in the class
$I^{r'}(X, \Gamma(\Sigma))$,
where
\[
r'= \frac{m+l}{2}.
\]
Taking into account the powers of $\h$ in front of the integral, we obtain that
\[
F(\Upsilon)\in I^{r+r'-(\dim X+N)/2} (X, \Gamma(\Sigma)).
\]
However,
by (\ref{theExcess})
\[
r' - N/2 = \frac 12\left( m+l-N\right) = \frac 12\left( \dim Y - e\right)
\]
and therefore
\[
F(\Upsilon)\in I^{r+(\dim Y -\dim X-e)/2} (X, \Gamma(\Sigma)),
\]
which is what we had to prove.

\section{An example of an FIO of isotropic type}

Let $X$ be a compact Riemannian manifold, $\Delta$ its Laplace-Beltrami operator,
and $P= \frac 12\h^2\Delta +V$ be a Schr\"odinger operator where $V\in C^\infty(X)$.
Denote by
\[
H: T^*X\to\bbR,\quad H(x,\xi) = \norm{\xi}^2+V(x)
\]
its principal symbol, and let
\[
P\psi_j = E_j \psi_j, \quad E_1\leq E_2\leq\cdots
\]
where $\{\psi_j\}$ is an orthonormal basis of $L^2(X)$.

Note that the $\psi_j$ and $E_j$  depend on $\h$, but we
suppress that dependence from the notation for simplicity.

\begin{theorem}\label{anHermite}
	Assume that zero is a regular value of $H$, and let
	$\Theta := H^{-1}(0)$.  Let $\varphi\in\calS(\bbR)$ be a
	Schwartz function, and define
	\[
	P_\varphi := \varphi\left(\frac{1}{\sqrt{\h}}\,P\right)
	\]
	by the spectral theorem so that
	$P_\varphi(\psi_j) = \varphi\left(\frac{E_j}{\sqrt{\h}}\right)\,\psi_j$ for all $j$.
	Then the Schwartz kernel of $P_\varphi$ is an isotropic function associated with
	the following isotropic submanifold of $T^*X\times T^*X$:
	\begin{equation}\label{diagTheta}
		\Theta\buildrel\Delta\over\times\Theta:= \{
	(x,x, \xi, -\xi)\:|\; (x,\xi)\in \Theta\}.
	\end{equation}
\end{theorem}

\begin{remark}
This  result is to be contrasted with:
	\begin{enumerate}
		\item[(a)] the kernels of operators of the form $\varphi(P)$ which are $\h$-pseudodifferential	operators, and
		\item[(b)]  operators of the form
		\begin{equation}\label{tfAssumptionOnPhi}
			\varphi\left(\frac{1}{{\h}}\,P\right) \quad\text{with}\quad \hat\varphi\in C_0^\infty(\bbR),
		\end{equation}	
		which are semi-classical Fourier integral operators associated to the null-leaf relation
		of $\Theta$.
	\end{enumerate}
\end{remark}

In the proof of the theorem we will need:
\begin{lemma}\label{needed1}
	Let $\chi\in C_0^\infty(\bbR)$ be a function that is identically equal to one in a neighborhood of the origin,
	and let $\chi_c(t) = 1-\chi(t)$.  Let $\rho\in\calS(\bbR)$, and consider
	\begin{equation}\label{lasGamma}
		\gamma(\lambda,\h) = \int e^{i\lambda s} \rho(s/\sqrt{\h})\, \chi_c(s)\, ds.
	\end{equation}
	Then
	\begin{equation}\label{}
		\forall k\geq 0,\,N>0\  \exists C>0\quad\text{such that}\quad \forall \lambda>0, \ \h\in(0,1) \text{\ we have }		\left|\gamma(\lambda,\h)\right| \leq C\,\lambda^{-k}\h^N .
	\end{equation}
\end{lemma}
\begin{proof}  The technique of proof is standard.
	Let $D_s = \frac{1}{i}\partial_s$, and write $e^{is\lambda} = \lambda^{-k}D_s^k(e^{is\lambda})$.
	Integrating by parts one gets
	\[\aligned
	\gamma(\lambda,\h) &= (-1)^k\,\lambda^{-k} \int e^{i\lambda s} D_s^k\left(\rho(s/\sqrt{\h})\, \chi_c(s)\right)\, ds \\
	 &	= i^k\,\lambda^{-k} \sum_{j=0}^k {k\choose j}\h^{-j/2}\int e^{i\lambda s} \rho^{(j)}(s/\sqrt{\h}) \chi_c^{(k-j)}(s)\, ds.
	\endaligned\]
	However, for each $j$ there exists $C>0$ such that
	\[
	\left|\int e^{i\lambda s} \rho^{(j)}(s/\sqrt{\h})\chi_c^{(k-j)}(s)\, ds\right| \leq
	C\int_{|s|\geq 1} \left|\rho^{(j)}(s/\sqrt{\h}) \right|\, ds =
	\sqrt{h}\, C \int_{|u|\geq \h^{-1/2}} \left|\rho^{(j)}(u) \right|\, du.
	\]
	Moreover, since $\rho$ is Schwartz, for any $N>0$ there is a constant $C$ such that for all $u$, $|u|\geq 1$
	implies
	$\left|\rho^{(j)}(u) \right| \leq C  |u|^{-N}$.  Therefore if $\h\in(0,1)$,
	\[
	\int_{|u|\geq \h^{-1/2}} \left|\rho^{(j)}(u) \right|\, du \leq C \int_{|u|\geq \h^{-1/2}} |u|^{-N}\, du =
	C'\h^{(N-1)/2}.
	\]
\end{proof}

\noindent {\em Proof of Theorem \ref{anHermite}.}   One has:
\[
P_\varphi = \frac{1}{2\pi}\int_\bbR e^{itP/\sqrt{h}}\, \hat\varphi(t)\, dt,
\]
where $\hat\varphi$ is the ordinary Fourier transform of $\varphi$.  Making the
substitution $t = s/\sqrt{\h}$ gives
\[
P_\varphi = \frac{1}{2\pi\sqrt{\h}}\int_\bbR e^{isP/\h}\, \hat\varphi(s/\sqrt{\h})\, ds.
\]
Let $\chi\in C_0^\infty(\mathbb R)$ be identically equal to one in a neighborhood of the origin.
Let us write
\begin{equation}\label{pPhi}
	P_\varphi = Q+R\quad\text{where}\quad Q=\frac{1}{2\pi\sqrt{\h}}\int_\bbR e^{isP/\h}\, \hat\varphi(s/\sqrt{\h})\, \chi(s)\,ds,
\end{equation}
and
\begin{equation}\label{}
	R = \frac{1}{2\pi\sqrt{\h}}\int_\bbR e^{isP/\h}\, \hat\varphi(s/\sqrt{\h})\, \chi_c(s)\,ds
	\quad\text{where}\quad \chi_c=1-\chi.
\end{equation}
The Schwartz kernel of $R$ is
\begin{equation}\label{scOfR}
	\calK_R(x,y) = \frac{1}{2\pi\sqrt{\h}}\sum_j \gamma(\hinv E_j,\h) \,\psi_j(x)\,\overline{\psi_j}(y),
\end{equation}
where $\gamma(\hinv E_j,\h)$ is given by (\ref{lasGamma}) with $\rho = \hat\varphi$.  By the
previous Lemma
\[
\forall k,\,N\ \exists C\quad\text{such that}\quad \forall j,\, \forall\h\in(0,1),\quad
\left| \gamma(\hinv E_j,\h)\right|\leq
C\,|E_j|^{-k}\,\h^N.
\]

\newcommand{\Vmin}{V_{\text{min}}}
We want to show that $R$ is a residual operator, that is, that $\calK_R$ is
$O(\h^\infty)$ with all its derivatives.
To see this, let us break (\ref{scOfR}) into two sums,
\[
\mathrm{(I)} = \frac{1}{2\pi\sqrt{\h}}\sum_{\substack{j \\ E_j >|V_{\text{min}}|}}
\gamma(\hinv E_j,\h) \,\psi_j(x)\,\overline{\psi_j}(y)
\]
and $\mathrm{(II)} = \calK_R(x,y)-\mathrm{(I)}$, where $V_{\text{min}}<0$ is the minimum of the potential.

To estimate $\mathrm{(I)}$ we use that, by the min-max principle
\begin{equation}\label{byMinMax}
\forall j,\,\forall\h,\qquad E_j \geq \h^2 \lambda_j + \Vmin,
\end{equation}
where $0<\lambda_1\leq \lambda_2\leq\cdots$ are the eigenvalues of $\Delta$ listed with multiplicities.
By the Weyl law for the Laplacian, $\exists A>0$ such that
$\forall \lambda,\ \  \#\{\ell\;|\; \lambda_\ell< \lambda\} \leq A\lambda^{n/2}$.
Let $j\in\bbN$ and let $\lambda = \lambda_j$ in this inequality to obtain that
\[
j \leq A\lambda_j^{n/2},\quad\text{or}\quad \lambda_j \geq Bj^{2/n}
\]
for some constant $B>0$.  Combining this with (\ref{byMinMax}) we conclude that
\[
\forall j,\,\h\qquad E_j \geq \h^2 Bj^{2/n} + \Vmin.
\]
Now the summation in $\mathrm{(I)}$ is over indices $j$ such that $E_j>-\Vmin$.  Adding these inequalities
we get that in all the terms in $\mathrm{(I)}$ one has
\[
\hinv E_j \geq \frac{B}{2} \h j^{2/n}.
\]
%
and therefore $\forall k,\,N,\ \exists C$ such that
\[
\left| \gamma(\hinv E_j,\h)\right|\leq C \h^N\,j^{-k}.
\]
This implies that $\mathrm{(I)}$ is 
$O(\h^\infty)$ together with all its derivatives and can be neglected.

To estimate \[
\mathrm{(II)} = \frac{1}{2\pi\sqrt{\h}}\sum_{\substack{j \\ E_j \leq |V_{\text{min}}|}}
\gamma(\hinv E_j,\h) \,\psi_j(x)\,\overline{\psi_j}(y)
\]
(whose number of summands is $O(\h^{-n})$) we use the
$L^\infty$ bounds on eigenfunctions by $O(\h^{-n})$, together with
the previous lemma with $k=0$ and $N$ arbitrary.  Again $\mathrm{(II)}$ can be neglected.

\smallskip
Now the Schwartz kernel, $\calU_P$,
of $e^{isP/\h}$ (regarded as a distribution on $\bbR\times X\times X$)
is a Lagrangian distribution associated to
\[
\Gamma := \{(s,\delta s; \phi_s(\x)'; \x)\in T^*\bbR\times T^*X\times T^*X\;;\; \delta s = -H(\x)\},
\]
where $\phi_s$ is the Hamilton flow of $H$
and the prime denotes the usual map $(y,\delta y)' =(y, -\delta y)\in T^*X$
(and we have separated points in different cotangent bundles by a semicolon).  Equation (\ref{pPhi}) shows that the kernel
of $Q$ is the push-forward of $\varphi(s/\sqrt{\h})\chi(s)\,\calU_P$ to $X\times X$.
Therefore, by Theorem \ref{FIOinvThm} the Schwartz kernel of $Q$ is an isotropic function,
and it is easy to check that the associated isotropic is (\ref{diagTheta}).

\hfill$\Box$

Taking the trace of $P_\varphi$ leads to the following result, which can also
be proved directly:
\begin{proposition}\label{prop:funnyTF}  Under the previous assumptions, for any Schwartz function $\varphi$,
\begin{equation}\label{funnyTF}
		\sum_j\varphi\left(\frac{E_j}{\sqrt{\h}}\right) \sim
		\frac{1}{(2\pi)^{n}}\,\h^{-n+\frac 12}\, |\Theta|
	\hat{\varphi}(0)+ \h^{-n+\frac 12}\sum_{j=1}^\infty \h^{j/2}c_j,
\end{equation}
as $\h\to 0$, where $|\Theta|$ is the Liouville measure of $\Theta$.
\end{proposition}
We emphasize that the assumption that $\hat{\varphi}$ be compactly supported is not needed,
in contrast to (\ref{tfAssumptionOnPhi}).
\begin{proof}
	We first note that the left-hand side is the trace of $P_\varphi$, and
	the operation of taking the trace is a Fourier integral operator from $X\times X$
	to a point.  Therefore, by Theorem \ref{FIOinvThm} the trace has an asymptotic expansion in powers of $\sqrt{\h}$.
	
To find the leading term, let us use (\ref{pPhi}) together with a known FIO approximation to the Schwartz kernel of $\calU_P$ for small time, yields
that the Schwartz kernel of $P_\varphi$ is locally  well-approximated by integrals of the form
\begin{equation}\label{s-kerPphi}
\calV(x,y,\h)=\frac{1}{(2\pi)^{n+1}\,\h^{n+\frac 12}}\int e^{i\hinv[S(s,x,p)-y\cdot p]} a(s,x,y,p,\h)\hat\varphi(s/\sqrt{\h})\,\chi(s)\, dp\,ds,
\end{equation}
where $S$ solves the Hamilton-Jacobi equation
\[
H(x,\nabla_xS(s,x,p)) + \frac{\partial S }{\partial s}(s,x,p)=0,\quad S(0,x,p)=x\cdot p
\]
and $a$ is an amplitude such that $a|_{s=0}=1$.
The result follows from a stationary phase expansion of the integrals
\[
\int\calV(x,y,\h)\,dx=\frac{1}{(2\pi)^{n+1}\,\h^{n+\frac 12}}\int e^{i\hinv[S(s,x,p)-x\cdot p]} a(s,x,x,p,\h)\hat\varphi(s/\sqrt{\h})\,\chi(s)\, dp\,ds\, dx,
\]
or equivalently,
\[
\int \calV(x,x,\h)\, dx =
\frac{1}{(2\pi)^{n+1}\,\h^{n}}\int e^{i\hinv[S(t\sqrt{\h},x,p)-x\cdot p]} a(t\sqrt{\h},x,x,p,\h)\hat\varphi(t)\,\chi(t\sqrt{\h})\, dp\,dt\,dx.
\]
In order to achieve this we need first to Taylor expand the phase
in the $s$ variable at $s=0$:
\[
S(t\sqrt{\h},x,p) = S(0,x,p)+ t\sqrt{\h}\, \frac{\partial S}{\partial s}(0,x,p) + \h t^2 T(t,x,p,\h).
\]
Then, noting that $S_s(0,x,p) = -H(x,\nabla_xS(0,x,p))$ and $\nabla_xS(0,x,p) = p$, we get
\[
\int \calV(x,x,\h)\, dx =
\frac{1}{(2\pi)^{n+1}\,\h^{n}}\int e^{-i\h^{-1/2}tH(x,p)} e^{it^2 T(t,x,p,\h)}a(t\sqrt{\h},x,x,p,\h)\hat\varphi(t)\,\chi(t\sqrt{\h})\, dp\,dt\,dx.
\]
It is clear from the $dt$ integral that we can cut-off the integrand away from $\Theta =\{H(x,p)=0\}$.
Since zero is a regular value of $H$, in a neighborhood of $\Theta$ we can use $H$ as a coordinate in phase space.
Let $(x,p)\mapsto (H,\zeta)$ be such a coordinate change, and let $dxdp = dH\wedge d\zeta$ where $d\zeta$ represents a form of degree $2n-1$
that pull-backs to Liouville measure on $\Theta$.  Then the integral above is of the form
\[
\int \calV(x,x,\h)\, dx = \frac{1}{(2\pi)^{n+1}\,\h^{n}}\int e^{-i\h^{-1/2}tH}A(t,H,\zeta,\h)\,\hat\varphi(t)\,\chi(t\sqrt{\h})\, dH\,d\zeta\,dt + O(\h^\infty).
\]
Applying the method of stationary phase in the $dt\,dH$ integral, which now has a quadratic phase, yields
\[
\int \calV(x,x,\h)\, dx \sim \hat{\varphi}(0)\frac{\h^{1/2}}{(2\pi\h)^n} \int A(0,0,\zeta,0) d\zeta .
\]
\end{proof}

From this result, a standard argument involving approximating
the characteristic function $\chi_c$
of $[-c,c]$ by functions $\varphi_\epsilon\in C_0^\infty(\bbR)$ 
yields the following:
%
%
%
\begin{corollary}\label{cor:altWeyl}
With the previous notation and assumption, $\forall c>0$,
 \begin{equation}\label{altWeyl}
 	\#\{\,j\;; |E_j|\leq c\,\h^{1/2}\,\} \sim 2c\, \frac{1}{(2\pi)^n}\, \h^{-n+\frac 12}\, |\Theta|.
 \end{equation}
\end{corollary}

\begin{remark}
	It is straightforward to replace in the proof of Proposition \ref{prop:funnyTF}
	the exponent $1/2$ by any $\alpha\in (0,1)$, yielding the estimate
	\begin{equation}\label{}
		\#\{\,j\;; |E_j|\leq c\,\h^\alpha\,\} \sim 2c\, \frac{1}{(2\pi)^n}\, \h^{-n+\alpha}\, |\Theta|.
	\end{equation}
By the semiclassical trace formula, \cite{PU}, the result remains true for $\alpha=1$ {\em provided} one assumes
that the set of periodic trajectories has Liouville measure equal to zero
(no such assumption is needed for $\alpha\in (0,1)$).

\end{remark}

\section{Superpositions of coherent states}
By definition, a coherent state is an isotropic function where the isotropic is a single point.
In this section we show that {\em every} isotropic function is a ``superposition of coherent states",
in a sense to be made precise.  We also prove that, conversely, certain superpositions of coherent
states yield isotropic functions.

\subsection{Every isotropic function is a superposition of coherent states}

As mentioned, the elements of $I^k(X, \Sigma)$ are \emph{coherent states} if the isotropic submanifold, $\Sigma$, of $T^*X$ is just a single point $(x_0, \xi_0)$, which will be referred to as the center of the coherent state.
More precisely, a {coherent state} centered at $(x_0, \xi_0) \in  T^*\mathbb R^n$ is a (semiclassical) function defined on $\mathbb R^n$ of the form
\begin{equation}\label{coherent}
\Upsilon_{x_0, \xi_0}(x) = e^{\sqrt{-1}f(x)/\hbar}a(\hbar^{-1/2}(x-x_0), \hbar),
\end{equation}
where $f$ is a smooth function with the property that
\[df(x_0)=\xi_0,\]
and $a(x,\h)$ is a Schwartz function in $x$ having an asymptotic expansion in $\hbar$ as in
\S 2.2. The classical example are the Gaussian functions
\[e^{(-|x-x_0|^2+i(x-x_0)\cdot \xi_0)/\hbar}.\]

If $\Upsilon$ is a coherent state centered at $(x_0, \xi_0)$, then its semiclassical wave front set is $\mathrm{WF}_\hbar(\Psi_{x_0, \xi_0})=\{(x_0, \xi_0)\}$. Moreover, if  $F$ is a semiclassical Fourier integral operator associated to a canonical transformation that maps $(x_0, \xi_0)$ to $(x_1, \xi_1)$, then $F(\Psi_{x_0, \xi_0})$ is a coherent state centered at $(x_1, \xi_1)$. As a consequence, one can define coherent states on manifolds as follows:
Let $X$ be a smooth manifold, and $(x_0, \xi_0) \in T^*X$ a point in the cotangent space. A coherent state $\Psi_{x_0, \xi_0}$ centered at $(x_0, \xi_0)$ is a function on $X$ so that
\begin{itemize}
	\item in some coordinate neighborhood $U$ of $X$ near $x_0$, $\Psi$ can be written as a function of the form (\ref{coherent});
	\item in some open set $V$ of any $x \ne x_0$, $\Psi$ and all its derivatives is of $O(\hbar^\infty)$, uniformly on compact sets.
\end{itemize}
It is immediate that the previously mentioned properties of coherent states in Euclidean space extend to manifolds.

We now prove:
\begin{proposition}\label{Iso=SupCoh}
	If $\Sigma \subset T^*X$ is an isotropic submanifold, every oscillatory function, $\Upsilon \in I^r(X, \Sigma)$ is
a superposition of coherent states supported at points $(x_0, \xi_0) \in \Sigma$.
	Specifically, given $\Upsilon\in I^r(X,\Sigma)$,
	there is a smooth family of coherent states $\Upsilon_\tau$, indexed by points $\tau\in\Sigma$, and
	a 1-density $d\tau$ on $\Sigma$ such that
	\begin{equation}
	\Upsilon = \int_\Sigma \Upsilon_\tau\,d\tau.
	\end{equation}
\end{proposition}

\begin{proof} Using a microlocal partition of unit,
	to prove Proposition \ref{Iso=SupCoh} it suffices to prove the result in the case
	when $\Sigma$ is the model isotropic, $\Sigma_0$, in (\ref{basicdef}).
	Namely,
	there exists a neighborhood, $\mathcal U$, of $p$ in $T^*X$, a neighborhood, $\mathcal U_0$, of $p_0$ in $T^*\mathbb R^n$ and a symplectomorphism, $\varphi: \mathcal U \to \mathcal U_0$ mapping $\Sigma \cap \mathcal U$ onto $\Sigma_0 \cap \mathcal U_0$. Moreover, by a microlocal partition of unity we can assume without loss of generality that the $\Upsilon$ in Proposition \ref{Iso=SupCoh} is supported in such a neighborhood and is, de facto, a $\Upsilon$ in $I^r(\mathbb R^n, \Sigma_0)$.

	Let $\Upsilon(x, \hbar)=\varphi(x', \hbar^{-1/2}x'', \hbar)$ be  an isotropic state as given in (\ref{basicdef}). For each $\tau \in \mathbb R^n$,  let
	\[
	\Upsilon_{\tau}:=
	\varphi(x', \hbar^{-1/2}x'', \h)e^{-\norm{x'-\tau}^2/\hbar}.
	\]
	This is a coherent state centered at $(\tau, 0; 0,0)\in\Sigma_0$, and
	\[
	\Upsilon = \frac{1}{(\pi\h)^{n/2}}\int_{\Sigma_0} \Upsilon_\tau\,d\tau.
	\]
	
%
\end{proof}

\subsection{``Coherent" superposition of coherent states yield isotropic functions}

Our goal is to make precise and prove the statement which is the title of this section.  The Bohr-Sommerfeld condition
plays a key role to define ``coherent".

\begin{definition}
	An isotropic submanifold $\Sigma\xrightarrow{\gamma} T^*X$ satisfies the Bohr-Sommerfeld
	condition iff $\gamma^*(pdx)$, where $pdx$ is the natural one-form on $T^*X$, has
	periods in $2\pi\bbZ$.
\end{definition}

\begin{lemma}
	let $\gamma: \Sigma\to T^*X$ be an isotropic embedding. Assume that $\gamma$ is horizontal, that is,
	assume that the composition
	\[
	\Sigma\xrightarrow{\gamma} T^*X \xrightarrow{\pi} X
	\]
	is a diffeomorphism onto an embedded submanifold $M\subset X$.  Then  there exists
	a horizontal Lagrangian $\Lambda\subset T^*X$ containing $\gamma(\Sigma)$
	(a horizontal framing).
\end{lemma}
\begin{proof}
	Let $s: M\to T^*X$ be the inverse of the projection $\gamma(\Sigma)\to M$.
	To each $x\in M$, $s$ associates a covector to $X$,
	$s(x)\in T_x^*X$.  Composing with the restriction maps
	\[
	\forall x\in M\qquad T^*_xX \to T^*_xM,
	\]
	$s$ defines a one-form $\sigma$ on $M$.  One can check that the
	condition that the image of $s$ is isotropic
	is equivalent to $d\sigma = 0$.
	
	{\sc Claim:} There is a neighborhood $\calU$ of $M$ in $X$ and a {\em closed} one-form $\beta\in \Omega^1(\calU)$ such that $\forall x\in M$, $s(x) = \beta_x$.
	
	To prove the claim we let $\calU$ be a tubular neighborhood of $M$, with projection
	$\frakp: \calU\to M$.  Let $\alpha = \frakp^*\sigma$.  Then $\alpha$ is closed
	and agrees with $s$ on vectors tangent to $M$.  We now show that one can
	modify $\alpha$ by adding the differential
	of a function $f:\calU\to\bbR$ so that $\beta = \alpha+df$ has the desired
	property.  It is enough to require that
	\begin{equation}\label{theProblem}
	f|_M \equiv 0\quad\text{and}\quad \forall x\in M, \, \forall v\in \ker(d\frakp_x)\ df_x(v) = s_x(v).
	\end{equation}
	
	Let $\{V_i\}$ be a locally finite cover of $M$ such that there exist trivializations
	$\frakp^{-1}(V_i) \cong V_i\times\bbR^\nu$.  For each $i$ it is easy to construct
	$f_i: \pi^{-1}(V_i)\to \bbR$ solving (\ref{theProblem}) on $V_i$.  Let $\{\chi_i\}$ be
	a  partition of unit subordinate to $\{V_i\}$ and let
	\[
	f = \sum_i \pi^*(\chi_i)\, f_i.
	\]
	It is easy to see that $f$ has the desired property.  Now let $\Lambda$ to be the image of
	the one-form $\beta$.
\end{proof}

\begin{remark}\label{keyObs}
	If $\Sigma$ satisfies the Bohr-Sommerfeld condition so does the horizontal framing
	constructed above, since the Lagrangian retracts to the isotropic.
	The one-form $\beta$ is locally of the form
	$\beta = d\psi$
	where $\psi$ is a function that is defined modulo $2\pi\bbZ$, so that $e^{i\psi}$ is a
	well-defined function near $M$.
\end{remark}

Let $\gamma:\Sigma\to T^*X$ now be any isotropic embedding of a compact manifold.  Let us define
\begin{equation}\label{susPension}
\widetilde{\Sigma} := \{(t,0 ; \gamma(t))\;;\; t\in\Sigma\}\subset T^*(\Sigma\times X).
\end{equation}
It is easy to see that this is an isotropic submanifold of $T^*(\Sigma\times X)$, and it is
clearly horizontal. The image of the projection of $\widetilde{\Sigma}$ onto
$\Sigma\times X$ is the graph
\begin{equation}\label{graph}
M = \{(t, x(t))\;;\; t\in\Sigma,\ x(t)= \pi(\gamma(t))\}
\end{equation}
of $\pi\circ\gamma: \Sigma\to X$.
Moreover, $\Sigma$ satisfies the Bohr-Sommerfeld condition iff
$\widetilde{\Sigma}$ does.

Assume from now on that $\Sigma$ satisfies the Bohr-Sommerfeld  condition, and let us restrict
the values of $\h$ to the reciprocals of the natural numbers.  Let $\psi(t,x)$ be
a function defined on a neighborhood $\calU$ of the
graph $M$, such that the image of $d\psi$ is a
framing of $\widetilde{\Sigma}$.  By the Bohr-Sommerfeld  condition $\psi$ is defined only
modulo $2\pi\bbZ$, but
$e^{i\psi(t,x)}$ is well-defined on $\calV$.

\begin{proposition}
	Consider a family of coherent states in $L^2(X)$ along $\gamma$ of the form
	\begin{equation}\label{theExpression}
	\psi_{\gamma(t)}(x) = e^{i\hinv\psi(t,x)}\, a\left(t, \frac{x-x(t)}{\sqrt{\h}}\right),
	\end{equation}
	where $a(t,\cdot)$ satisfies the usual Schwartz estimates.
	Then the superposition of coherent states
	\begin{equation}\label{superP}
	\Upsilon := \int_\Sigma \psi_{\gamma(t)}\, dt
	\end{equation}
	is an isotropic function on $X$ associated with the image of $\gamma$.
\end{proposition}
\begin{proof}
	Let us consider the entire family $\Psi=\{\psi_{\gamma(t)}\}_{t\in\Sigma}$ as a
	smooth function on $\Sigma\times X$.  Then one has that
	\begin{equation}\label{}
	\Upsilon = p_*(\Psi), \quad \text{where}\ p: \Sigma\times X\to X\ \text{is the projection}.
	\end{equation}
	By inspecting the expression (\ref{theExpression}) and taking into account (\ref{graph}), we can conclude
	that
	$\Psi$ is an isotropic function associated to $\widetilde{\Sigma}$.
	By functoriality of isotropic functions with respect to the action of Fourier integral
	operators (in this case the push-forward operator, $p_*$), $\Upsilon$ is an isotropic
	function given that $\widetilde{\Sigma}$ is contained in the conormal bundle to the fibers of $\pi$.
\end{proof}

\subsection{An application}
In this final section we briefly sketch some applications of our previous results.
Throughout we let $X$ be a compact manifold and $P$ a self-adjoint $\h$-pseudodifferential operator of order zero,
for example a Schr\"odinger operator, $P =\frac 12 \h^2\Delta + V$.  Let $H:T^*X\to \bbR$ be its principal symbol,
$H(x,\xi) = \frac 12 \norm{\xi}^2_x + V(x)$.  We denote by $U(t) = e^{-it\hinv P}$
the fundamental solution to the time-dependent Schr\"odinger equation
\[
i\h\frac{\partial\ }{\partial t} U(t)\psi_0 = PU(t)(\psi_0),\quad U(0) = I.
\]
This is a semi-classical Fourier integral operator, in the strong sense that its Schwartz
kernel $U(t,x,y)$ is a Lagrangian distribution on $\bbR\times X\times X$ associated with
the canonical relation
\begin{equation}\label{}
\Gamma = \{(t,\tau;\,x,\xi;\, y,\eta)\in T^*(\bbR\times X\times X)\;;\; (x, \xi) = \phi_t(y,\eta),\ \tau = H(x,\xi) \},
\end{equation}
where $\phi_t: T^*X\to T^*X$ is the Hamilton flow of $H$.


By the general theory, if $\psi_0\in C^\infty(X)$ is a coherent state centered at $(x_0,\xi_0)\in T^*X$, then
for each $t\in\bbR$, $U(t)(\psi_0)$ is a coherent state centered at $\phi_t(x_0,\xi_0)$.  Moreover, the estimates
implied by this statement are uniform provided $t,\,(x_0,\xi_0)$ take values in a compact set.

Consider now $\Sigma\subset T^*X$ a compact isotropic (possibly Lagrangian) submanifold, and let
$\Upsilon\in I^{r}(X,\Sigma)$ be an associated isotropic function.  By Proposition \ref{Iso=SupCoh}, there exists a smooth family
$\Upsilon_\sigma$ of coherent states, indexed by points $\sigma\in \Sigma$, such that
\[
\Upsilon = \int_\Sigma \Upsilon_\sigma\, d\sigma
\]
for some density $d\sigma$ on $\Sigma$.  Then we can write
\begin{equation}\label{}
U(t)(\Upsilon) = \int_\Sigma U(t)(\Upsilon_\sigma)\, d\sigma.
\end{equation}
This expresses $U(t)(\Upsilon)$ as a superposition of coherent states.
Thus, in principle, it suffices to propagate coherent states in order to
compute the propagation of isotropic states.


\end{document}